\documentclass[11pt]{amsart}
\usepackage{bm,amssymb,amscd,url,mathrsfs}
\usepackage{graphics,epsfig,color}

\usepackage[letterpaper,asymmetric,left=1.0in,right=1.0in,top=1.0in,bottom=1.
0in,bindingoffset=0.0in]{geometry}

\newcommand{\MATRIXGP}{{\Gamma}}
\newcommand{\AUTOGROUP}{{\mathcal G}}

\newcommand{\tr}{{\rm tr}}
\newcommand{\Aut}{{\rm Aut}}
\newcommand{\Inn}{{\rm Inn}}
\newcommand{\Out}{{\rm Out}}
\newcommand{\F}{{\rm F}}
\newcommand{\SL}{{\rm SL}}
\newcommand{\C}{\mathbb{C}}

\newcommand{\Z}{\mathbb{Z}}

\newcommand{\discrete}{\mathcal{D}}
\newcommand{\nondiscrete}{\mathcal{N}}
\newcommand{\parameterneighborhood}{\mathcal{P}}

\newcommand{\ourboundary}{\mathscr{B}}

\newtheorem{theorem}{\bf Theorem}[section]
\newtheorem{proposition}[theorem]{\bf Proposition}

\newtheorem*{theorem*}{\bf Theorem}
\newtheorem*{theoremA}{\bf Theorem A}

\newtheorem{remark}[theorem]{\bf Remark}

\newtheorem{question}{\bf Question}

\begin{document}

\title[Questions about the dynamics on a natural family of affine cubic
surfaces]{Questions about the dynamics on a \\  natural family of affine cubic
surfaces}

\author{Julio Rebelo and Roland Roeder}
\date{\today}

\thanks{}


\maketitle

\begin{abstract}
We present several questions about the dynamics of
the group of holomorphic automorphisms of the affine
cubic surfaces
\begin{align*}
S_{A,B,C,D} = \{(x,y,z) \in \C^3 \, : \, x^2 + y^2 + z^2 +xyz = Ax + By+Cz+D\},
\end{align*}
where $A,B,C,$ and $D$ are complex parameters.  
This group action describes the monodromy of the famous Painlev\'e 6 Equation as well as the natural dynamics
of the mapping class group on the $\SL(2,\C)$ character varieties associated to the once punctured torus and the four times punctured sphere.

The questions presented here arose while preparing our work ``Dynamics of groups of automorphisms of character varieties and Fatou/Julia decomposition for Painlev\'e~6'' \cite{RR} and during informal discussions with many people.
Several of the questions were posed at the Simons Symposium on Algebraic, Complex and Arithmetic Dynamics
that was held at Schloss Elmau, Germany, in August 2022 as well as the MINT Summer School ``Facets of Complex Dynamics'' that was held in Toulouse, France, in June 2023.
\end{abstract}

\section{Introduction}
This paper is devoted to presenting
several questions about the dynamics of the group of
automorphism of the complex affine surface 
\begin{align}\label{EQN:SURF}
S_{A,B,C,D} = \{(x,y,z) \in \C^3 \, : \, x^2 + y^2 + z^2 +xyz = Ax + By+Cz+D\},
\end{align}
where $A,B,C,$ and $D$ are complex parameters.
In order
to present them in the clearest possible way, we will first introduce some of
the most relevant problems for specific choices of parameter that are of
independent interest, before considering the full complex four-dimensional
family $S_{A,B,C,D}$.   A more ``top down'' approach and considerably greater
detail is given in \cite{RR} as well as in \cite{cantat-2,cantat-1}.

\vspace{0.1in}
{\bf Some ``preferred'' parameter values:}
\begin{itemize}
\item[(i)] {\em Picard Parameter}: $(A,B,C,D) = (0,0,0,4)$,
\item[(ii)] {\em Markoff Parameter}: $(A,B,C,D) = (0,0,0,0)$, and the
\item[(iii)] {\em Punctured Torus Parameters}: $(A,B,C,D) = (0,0,0,D)$ for any $D \in \C$.
\end{itemize}
All of the dynamics of the automorphism group ${\rm Aut}(S_{A,B,C,D})$ of $S_{A,B,C,D}$ can
be described quite explicitly for the Picard Parameter and therefore they
serve as a kind of ``basepoint'' for the dynamics (akin to the role that
$z \mapsto z^2$ plays for the quadratic family). The corresponding surface $S_{0,0,0,4}$
has four singular points and it is called the Cayley Cubic.

The Markoff Parameter $(0,0,0,0)$ is of special interest because 
the dynamics on $S_{0,0,0,0}$ 
has been studied extensively due to deep connections
with number theory, Teichm\"uller Theory, and other areas; see, e.g.,\
\cite{bowditch} and \cite{series}.   However, unlike the Picard Parameter,
there are many difficult questions about the dynamics of $\AUTOGROUP$ for the
Markoff Parameter.  In fact, most, but not all, of the questions posed in this
note are relevant and interesting for the Markoff Parameter.

While the full four-dimensional family of parameters $(A,B,C,D) \in \C^4$ corresponds to
the relative character varieties (of representations in $\SL(2,\mathbb{C})$) of the two dimensional sphere with four punctures (holes),
the sub-family $(A,B,C,D) = (0,0,0,D)$ corresponds to the relative character varieties of
the two dimensional torus with one puncture (hole), hence the name {\em Punctured Torus
Parameters} (see Section \ref{SEC:CHARACTERDYNAMICS} below).  They are somewhat simpler to study than the general four-dimensional
family.  Most of this paper focuses on the one parameter family of punctured torus parameters, however it ends
with Section \ref{SEC:ARBITRARYPARAMS} about the full four dimensional family corresponding to arbitrary $(A,B,C,D)$.

Let us close this introduction by mentioning that readers who are more interested in the dynamics of
groups of holomorphic automorphisms of {\em compact surfaces} may wish to consult the paper \cite{CD} by Cantat and Dujardin 
in this same volume.

\section{Background and definitions}\label{SUBSEC:BACKGROUND_AND_DEFS}
For simplicity we will start with the Punctured Torus Parameters $(A,B,C,D) = (0,0,0,D)$, and,
to abridge notation, we will use a single subscript $D$
to denote which surface is being considered: 
\begin{align*}
S_D \equiv S_{0,0,0,D} = \{(x,y,z) \in \C^3 \, : \, x^2 + y^2 + z^2 +xyz = D\}.
\end{align*}
As $D$ ranges over $\C$, these surfaces provide a fibration of $\C^3$ with some of the surfaces
being singular.  For example, $S_0$ is singular at $(0,0,0)$ and $S_4$ is the famous Cayley Cubic, which is singular
at four points $\{(-2,-2,-2), \, (-2,2,2), \, (2,-2,2), \, (2,2,-2)\}$.

Every line parallel to the $x$-axis
intersects the surface $S_D$ at two points (counted with multiplicity) and one can therefore
define an involution $s_x: S_D \rightarrow S_D$ that switches them:
\begin{align}\label{EQN:DEFSX}
s_x\left(x, y, z\right) = \left(-x-yz,  y, z \right).
\end{align}
Two further involutions $s_y: S_D \rightarrow S_D$ and $s_z: S_D \rightarrow S_D$
are defined analogously by means of lines parallel to the $y$-axis and to the $z$-axis, respectively. Namely, we have
$s_y\left(x, y, z\right) = \left(x,  -y-xz, z \right)$ and $s_z\left(x, y, -z-xy \right)$.

Consider the group
\begin{align*}
\AUTOGROUP^\pm \equiv \AUTOGROUP^\pm_D = \langle s_x, s_y, s_z \rangle\leq {\rm Aut}(S_{D}),
\end{align*}
where ${\rm Aut}(S_{D})$ denotes the group of all (algebraic) holomorphic diffeomorphisms of $S_{D}$.
It turns out that the index of $\AUTOGROUP^\pm$ as subgroup of ${\rm Aut}(S_{D})$ is at most $24$; See \cite[Theorem 3.1]{cantat-2} and also \cite{HUTI}.

Consider also the subgroup $\AUTOGROUP \equiv \AUTOGROUP_{D}$ of $\AUTOGROUP^\pm$ which is defined by
\begin{align*}
\AUTOGROUP \equiv \AUTOGROUP_{D}= \langle g_x,  g_y, g_z \rangle  < \AUTOGROUP^\pm,
\end{align*}
where $g_x = s_z \circ s_y$, $g_y = s_x \circ s_z$, and $g_z = s_y \circ s_x$.

\begin{remark}
The formulae for the elements of $\AUTOGROUP^\pm$ and, in particular, of $\AUTOGROUP < \AUTOGROUP^\pm$, do not
depend on the parameter $D$ and, on occasion, it is helpful to consider these
groups acting on all of $\C^3$.  However, for most of the time we focus on the action of these groups on each particular surface~$S_D$.
\end{remark}

\subsection{Invariant holomorphic two-form}\label{SEC:INV_FORM}
The smooth part of $S_{D}$
comes equipped with a holomorphic volume form
\begin{align}\label{EQN:VOLUME_FORM}
\Omega = \frac{dx \wedge dy}{2z+xy} = \frac{dy \wedge dz}{2x+yz} = \frac{dz \wedge dx}{2y + zx}
\end{align}
and a simple calculation shows that $s_x^* \Omega = -\Omega$ and similarly for
$s_y$ and $s_z$.  Therefore,  $\AUTOGROUP$ could also have been defined more intrinsically as the index two subgroup
of $\AUTOGROUP^\pm$ that preserves $\Omega$. 
Hence all elements of $\AUTOGROUP^\pm$ preserve the associated real volume form $\Omega \wedge
\overline{\Omega}$.  The latter assigns infinite volume to the whole surface
$S_{D}$ but a finite volume to a sufficiently small neighborhood of each singular point of~$S_{D}$.

\subsection{Fatou/Julia Dichotomy}\label{SEC:FATOU_JULIA}
The {\em Fatou set} of the group action $\AUTOGROUP$ is defined as
\begin{align*}
{\mathcal F}_{D} = \{p \in S_{D} \, : \, \mbox{$\AUTOGROUP$ forms a normal family in some open neighborhood of $p$}\}.
\end{align*}
Here, the term {\em normal} is interpreted in the sense of Montel's Theorem:
given an open set $U$ we require for any sequence of mappings
$(\gamma_{n})_{n=1}^\infty  \subset \AUTOGROUP$ that there be a subsequence
$(\gamma_{n_j})_{j=1}^\infty$ such that $\gamma_{n_j}$ converges uniformly on
compact subsets of $U$. Because $S_D$ is non-compact, the definition of {\it convergent sequence on a compact
set $\overline{V} \subset S_D$} needs to include the possibility that $(\gamma_{n_j})_{j=1}^\infty$
{\em diverges properly to infinity}, i.e.\ for any compact  $K \subset S_D$ there are at most
finitely many indices $j$ such that $\gamma_{n_j}\left(\overline{V}\right) \cap K \neq
\emptyset$.

\begin{remark}
The above definition of Fatou set has the advantage of being easy to state
while still yielding interesting results.   However, there are other possible definitions
that might be worthy of consideration.

One possible alternative is to let $S_D' = S_D \cup \{\infty \}$ denote the
one-point compactification of $S_D$ and to consider the elements of
$\AUTOGROUP$ as mappings $\gamma: S_D \rightarrow S_D'$.  In the definition of
{\em normal} one would then require the subsequence
$(\gamma_{n_j})_{j=1}^\infty$ (as in the previous paragraph) to converge
uniformly on compact sets of $U$ to a mapping $\Upsilon: U \rightarrow S_D'$.
A-priori, this alternative definition could lead to a bigger Fatou set
${\mathcal F}'_D$ than ${\mathcal F}_D$ (defined above)
because of the possibility that the image of the limit map $\Upsilon(U)$ contains
both $\infty$ and points of $S_D$. This possibility, however, can be ruled out by showing that
${\mathcal F}'_D$ is Kobayashi hyperbolic, using essentially the same proof as
for Proposition 9.1 from \cite{RR} where it was shown that ${\mathcal F}_D$ is Kobayashi hyperbolic.
Therefore ${\mathcal F}'_D$ = ${\mathcal F}_D$.
\end{remark}

The {\em Julia set} of the group action $\AUTOGROUP$ is defined as
\begin{align*}
{\mathcal J}_{D} = S_{D} \setminus {\mathcal F}_{D}.
\end{align*}
It follows from the definitions that ${\mathcal
F}_{D}$ is open while ${\mathcal J}_{D}$ is closed.  Furthermore
both sets are invariant under $\AUTOGROUP$.
It is rather easy to see that ${\mathcal J_D} \neq \emptyset$ for any $D \in
\C$; see e.g.\ \cite[Lemma 4.3]{RR}.  With more effort one can show that for every parameter $D$
the group $\AUTOGROUP_D$ contains elements $f_D$ having positive entropy and exhibiting dynamics quite
similar to that of H\'enon maps; see \cite{IU_ERGODIC,cantat-1,cantat-2}.   The Julia set associated with iteration of each
such individual mapping is a subset of ${\mathcal J}_D$, including the collection of all saddle-type periodic points
of each such mapping.

One motivation for defining this dichotomy is the following theorem from \cite{RR}:

\begin{theoremA}
For any parameter $D \in \mathbb{C}$ there is a point $p \in {\mathcal J}_{D}$ whose
orbit under $\AUTOGROUP$ is dense in~${\mathcal J}_{D}$.
\end{theoremA}

(In fact, this theorem holds for the full four-dimensional family $S_{A,B,C,D}$.)

\begin{remark}
One might be concerned about applying the Fatou/Julia dichotomy in a neighborhood of a singular point of $S_D$.
However, it is relatively easy to check that any such point must be in the Julia set; see \cite[Remark 3.1]{RR}.
\end{remark}

\begin{remark} 
For any parameter $(A,B,C,D)$ any singular point of $S_{A,B,C,D}$ is a quotient
singularity; See, e.g. \cite[Sec. 3.4]{cantat-2}.   Therefore, $S_{A,B,C,D}$
has the structure of a smooth orbifold and the group $\AUTOGROUP_{A,B,C,D}$ acts by
holomorphic self-maps of that orbifold.   Considering the Fatou/Julia dichotomy in this
broader context one again finds that any singular point is in the Julia set ${\mathcal J}_{A,B,C,D}$.
\end{remark}

\subsection{Locally Non-discrete/Locally Discrete Dichotomy}\label{SEC:DISC}
Let $M$ be a (possibly open) connected complex manifold and consider a group $G$ of holomorphic diffeomorphisms of $M$.
The group $G$ is said to be {\it locally non-discrete} on an open set $U \subset M$
if there is a sequence of maps $(f_n)_{n=0}^{\infty} \in G$ satisfying the following conditions (see for example \cite{REBELO_REIS}):
\begin{enumerate}
 \item For every $n$, $f_n$ is different from the identity.
  \item The sequence of maps $f_n$ converges uniformly to the identity on compact subsets of $U$.
\end{enumerate}
If there is no such sequence $f_n$ on $U$ we say that $G$ is {\it locally discrete} on $U$.

The {\em Locally Non-discrete Locus} for the action of $\AUTOGROUP$ is
\begin{flalign*}
\nondiscrete_{D} = \{p \in S_{D} \, : \, \AUTOGROUP_{D} \, \mbox{is locally non-discrete on an open neighborhood $U$ of $p$}\}.
\end{flalign*}
The {\em Locally Discrete Locus} is $\discrete_{D} =   S_{D} \setminus \nondiscrete_{D}$.
By definition, $\nondiscrete_{D}$ is open, $\discrete_{D}$ is closed, and each is invariant under $\AUTOGROUP_{D}$.

\begin{remark}
The interplay between the Fatou/Julia and Locally Non-discrete/Locally Discrete dichotomies
is subtle and rich, with no immediate relationship between them.
\end{remark}

\begin{remark}
The definition of locally non-discrete group can be formulated in broader contexts, including in
the setting of pseudo-groups and with weaker regularity conditions,
both on the mappings themselves and on the mode of convergence.  
To abridge the discussion here we will stick to the context of groups of globally defined homeomorphisms.

The present setting of biholomorphic mappings of a complex manifold is the simplest one
because locally uniform convergence implies $C^\infty$ convergence
by virtue of the Cauchy formula.
Local non-discreteness has also been considered beyond the holomorphic
setting, see for example \cite{GHYS,R1}. Some amount of regularity is usually needed to have simple criteria to
ensure local non-discreteness, in particular dealing with homeomorphisms in full generality may be tricky (see
nonetheless \cite{R2}).
\end{remark}

\section{Dynamics on the character variety of the one-holed torus.}
\label{SEC:CHARACTERDYNAMICS}
The interpretation of the action of $\AUTOGROUP_D^\pm$ in terms of the natural dynamics on the character variety of the two-dimensional torus with one puncture (hole) plays an important role throughout the paper, and especially in Section \ref{SUBSEC:FJ_MARKOFF}, below.
We refer the reader to \cite[Section 4]{bowditch} and \cite{GOLDMAN}, as well as the previously mentioned papers \cite{cantat-1,cantat-2,RR} for more details.
(This section can be skimmed over on a first reading and then returned to later as details are needed.)

Let $\mathbb{T}_0$ denote the topological two-dimensional torus with one puncture (hole).   Its fundamental
group satisfies $\pi_1(\mathbb{T}_0) \cong \F_2 \cong \langle a,b \rangle$, where $\F_2$ denotes
the free group on two generators $a$ and $b$.   Note that the commutator $aba^{-1}b^{-1}$ corresponds to a simple loop around the puncture.
Let
\begin{align*}
{\rm Rep}(\mathbb{T}_0) = \{\rho: \F_2 \rightarrow \SL(2,\C)\}
\end{align*}
be the space of all representations of $\pi_1(\mathbb{T}_0) \cong \F_2$
into $\SL(2,\C)$.  They are not assumed to be faithful or discrete and therefore
one has ${\rm Rep}(\mathbb{T}_0) \cong \SL(2,\C)^2$, with the isomorphism given
by $\rho \mapsto (\rho(a),\rho(b))$.  

The {\em Character Variety} is defined as
\begin{align*}
\chi = {\rm Rep}(\mathbb{T}_0) // \SL(2,\C),
\end{align*}
where the quotient corresponds to simultaneous conjugation by an element of $\SL(2,\C)$.  (The double $//$ denotes that one is doing the ``categorical'' quotient; for more details see, for example, \cite[Sec.~2]{BG}.)   For any fixed choice of $\tau \in \C$, one can also consider the {\em Relative Character Variety}
\begin{align*}
\chi_\tau = \{[\rho] \in \chi \, : \, {\rm tr}(\rho(aba^{-1}b^{-1}) = \tau\}.
\end{align*}
Following the work of Fricke one can parameterize $\chi$ using the traces of elements:
\begin{align}\label{EQN:FRICK_PARAM}
\chi \cong \mathbb{C}^3 \qquad \mbox{given by} \qquad [\rho] \mapsto \left(\tr(\rho(a)), \tr(\rho(b)), -\tr(\rho(ab))\right) = (x,y,z).
\end{align}
Using the Fricke trace relations one computes that
\begin{align*}
\tr(\rho(aba^{-1}b^{-1})) = x^2 + y^2+z^2 +xyz -2
\end{align*}
and therefore the relative character variety is parameterized as
\begin{align*}
\chi_\tau \cong \{(x,y,z) \in \C^3 \, : \, x^2 + y^2+z^2 +xyz -2 = \tau\} = S_{\tau+2}.
\end{align*}
In particular, the Markoff surface $S_0$ corresponds to imposing $\tr(\rho(aba^{-1}b^{-1})) = -2$ and the Cayley Cubic $S_4$ (Picard Parameter)
corresponds to imposing $\tr(\rho(aba^{-1}b^{-1})) = 2$.

\begin{remark}
It is more common to omit the minus sign on the last coordinate $\tr(\rho(ab))$ in {\rm (}\ref{EQN:FRICK_PARAM}{\rm )}.  Doing so leads to the more standard equation $x^2+y^2+z^2 -xyz-2= \tau$
for $\chi_\tau$.  We introduced the minus sign so that the coordinates $(x,y,z)$ satisfy the equation of $S_D$
which fits naturally into the four parameter family of surfaces $S_{A,B,C,D}$, whose defining equation is the conventional way to describe the character variety of the
four times punctured (holed) sphere.
\end{remark}

We now describe how the dynamics generated by the involutions $s_x, s_y,$ and $s_z$ enters the picture.
The automorphism group $\Aut(\F_2) = \Aut(\pi_1(\mathbb{T}_0))$ acts on
${\rm Rep}(\mathbb{T}_0)$ by precomposition.   Given $\phi \in \Aut(\F_2)$ we have
\begin{align*}
f_\phi: {\rm Rep}(\mathbb{T}_0) \rightarrow {\rm Rep}(\mathbb{T}_0)  \qquad \mbox{given by} \quad  f_\phi(\rho) =  \rho \circ \phi^{-1}.
\end{align*}
If $\phi$ is an inner automorphism of $\F_2$, then the induced action consists of
conjugation by an element of $\SL(2,\C)$.  Therefore this action of $\Aut(\F_2)$ on ${\rm Rep}(\mathbb{T}_0)$ factors through an
action of $\Out(\F_2) = \Aut(\F_2)/\Inn(\F_2)$ on $\chi$.

A theorem of Nielsen \cite{N} asserts that $\Out(\F_2) \cong {\rm GL}(2 , \Z) \cong {\rm Mod}^\pm(\mathbb{T}_0)$, where
${\rm Mod}^\pm(\mathbb{T}_0)$ denotes the extended mapping class group of $\mathbb{T}_0$.   It consists of the group of isotopy classes
of self-homeomorphisms of $\mathbb{T}_0$, including those that reverse orientation.   Using that $\Out(\F_2) \cong {\rm Mod}^\pm(\mathbb{T}_0)$
and the fact that a self-homeomorphism of $\mathbb{T}_0$ sends a small simple loop around the puncture to another small
simple loop around the puncture, possibly with the opposite orientation, 
one has that any $\phi \in \Out(\F_2)$ maps the commutator $aba^{-1}b^{-1}$ to (a conjugate) of itself or its inverse.
Since  similar matrices have the same trace and since any element of $\SL(2,\C)$ has the same trace as its inverse, we
conclude that the action of $\Out(\F_2)$ on $\chi$
preserves $\tr(\rho(aba^{-1}b^{-1}))$.  In particular $\Out(\F_2)$ acts on
$\chi_{-2} \equiv S_0$.

The action of $\Out(\F_2)$ on $\chi_{-2} \equiv S_0$  is
explicitly described as follows.   Given the equivalence class $[\phi] \in \Out(\F_2)$ of an automorphism $\phi \in \Aut(\F_2)$ 
we define
\begin{align}\label{EQN:INDUCEDMAPPING}
f_\phi(x,y,z) &= f_\phi \big(\tr(\rho(a)), \tr(\rho(b)), -\tr(\rho(ab)\big) \\ 
&= \big(\tr(\rho(\phi^{-1}(a))), \tr(\rho(\phi^{-1}(b))), -\tr(\rho(\phi^{-1}(ab))\big). \nonumber
\end{align}
Using the isomorphism ${\rm GL}(2 , \Z) \cong \Out(\F_2)$, by minor abuse of notation, we
assign to each $M~\in~{\rm GL} (2 , \Z)$ the automorphism 
\begin{align}\label{EQN:INDUCEDMAPPING2}
f_M(x,y,z) \equiv f_{\phi_M}(x,y,z),
\end{align}
where $[\phi_M] \in \Out(\F_2)$ corresponds to the $2\times 2$ matrix $M \in {\rm GL} (2 , \Z)$.

Consider the automorphism $\phi(a,b) = (a^{-1},b^{-1})$ associated with the matrix $-{\rm Id} \in {\rm GL}(2,\mathbb{Z})$.
For any $M \in {\rm GL} (2 , \Z)$ we have $\tr(M) = \tr\left(M^{-1}\right)$ and $\tr (MN) = \tr(NM)$ so that
$f_{-{\rm Id}}(x,y,z) = (x,y,z)$.  Therefore, the action of ${\rm GL}(2,\mathbb{Z})$ factors through an action of ${\rm PGL}(2,\mathbb{Z}) = {\rm GL}(2,\mathbb{Z}) / \left\{\pm {\rm Id}\right\}$ on $S_{D}$.

Rather than working with the action of all of ${\rm PGL}(2,\mathbb{Z})$ on $S_D \equiv \chi_{D-2}$ we focus on the
action of the following index $6$ subgroup:
\begin{align*}
\MATRIXGP^\pm_2 := \{M \in {\rm PGL}(2,\mathbb{Z}) \, : \, M \equiv {\rm Id} \,  {\rm mod} 2 \}.
\end{align*}
The group $\MATRIXGP^\pm_2$ is generated by the matrices:
\begin{align*}
M_x:= \left[\begin{array}{cc} -1 & -2 \\ 0 & 1 \end{array}\right], \qquad M_y:= \left[\begin{array}{cc} 1 & 0 \\ -2 & -1\end{array}\right], \qquad \mbox{and} \qquad
M_z:=\left[\begin{array}{cc} 1 & 0 \\ 0 & -1\end{array}\right].
\end{align*}
To see this, consider 
the more classical congruence subgroup 
\begin{align*}
\MATRIXGP_2 := \{M \in {\rm PSL}(2,\mathbb{Z}) \, : \, M \equiv {\rm Id} \,  {\rm mod} 2\},
\end{align*}
which admits the well-known generators
\begin{align*}
\left[\begin{array}{cc} 1 & 2 \\ 0 & 1 \end{array}\right] \quad \mbox{and} \quad  \left[\begin{array}{cc} 1 & 0 \\ 2 & 1\end{array}\right].
\end{align*}
See, for example, \cite[Section 16.3]{CONWAY} or Exercises 6 and 7 from \cite[Chapter 13]{ZAK}.   
The claim that $M_x, M_y$, and $M_z$ generate $\MATRIXGP^\pm_2$ then follows because
\begin{itemize}
\item[(i)] $\MATRIXGP_2$ is an index two subgroup of $\MATRIXGP_2^\pm$,
\item[(ii)] both of the generators of $\MATRIXGP_2$ can be expressed as products of $M_x, M_y,$ and $M_z$ (up to equivalence modulo multiplication by $-{\rm Id}$), and
\item[(iii)] the matrices $M_x, M_y,$ and $M_z$ are in $\MATRIXGP_2^\pm$ but not in $\MATRIXGP_2$.
\end{itemize}

These matrices $M_x, M_y$, and $M_z$ correspond to the elements of $\Out(\F_2)$ represented respectively by the automorphisms
of $\F_2$ given by $\phi_{M_x}(a,b) =
(a^{-1} b^{-2},b), \ \phi_{M_y}(a,b) = (a,a^{-2}b^{-1})$ and $\phi_{M_z}(a,b) = (a,b^{-1})$.

\begin{proposition}
\label{PROP:DERIVINGSXSYSZ}
We have:
\begin{align*}
f_{M_x}(x,y,z) &=  (-x-yz,y,z) = s_x(x,y,z), \\
 f_{M_y}(x,y,z) &= (x,-y-xz,z) = s_y(x,y,z), \quad \mbox{and} \\
f_{M_z}(x,y,z) &= (x,y,-z-xy) = s_z(x,y,z).
\end{align*}
\noindent
For any parameter $D \in \mathbb{C}$ 
the assignments $M_x \rightarrow s_x$, $M_y \rightarrow s_y$, and $M_z \rightarrow s_z$ induce an isomorphism
of groups $\MATRIXGP^\pm_2 \rightarrow \AUTOGROUP_D^\pm$.
\end{proposition}
We refer the reader to \cite[Appendix]{goldman-1} and \cite[Section
3.2]{cantat-2}.  For the convenience of the reader we also include a brief
proof of this proposition in Appendix \ref{SEC:DERIVINGSXSYSZ}.

\section{Picard Parameter $D=4$}\label{SEC:PICARD}
For parameter $D = 4$ the dynamics of $\AUTOGROUP_D$ can be described quite explicitly.
Associated with a matrix $M = [m_{ij}] \in {\rm GL}(2,\mathbb{Z})$ is a monomial mapping 
\begin{align*}
\eta_M: \mathbb{C}^* \times
\mathbb{C}^* \rightarrow \mathbb{C}^* \times \mathbb{C}^* \qquad \mbox{given by} \qquad \eta_M \left(u, v \right) = \left(u^{m_{11}} v^{m_{12}}, u^{m_{21}} v^{m_{22}}\right).
\end{align*}
Consider also the mapping $\Phi: \mathbb{C}^* \times \mathbb{C}^* \rightarrow S$ given by
\begin{align*}
 \Phi(u,v) = \left(-u - 1/u, -v - 1/v, -u/v - v/u \right),
\end{align*}
which 
is a degree two ramified cover.  It is the quotient map for the action of the central monomial involution $(u,v) \mapsto (1/u,1/v)$.

\begin{proposition}\label{PROP:SEMI_CONJUGACY}
$\Phi$ semi-conjugates the action of ${\rm GL}(2,\mathbb{Z})$ on $\mathbb{C}^* \times \mathbb{C}^*$ by monomial mappings to the action of
${\rm GL}(2,\mathbb{Z})$ on $S_4$ described in  {\rm (}\ref{EQN:INDUCEDMAPPING2}{\rm )}.
More specifically, given $M \in {\rm GL}(2,\mathbb{Z})$ and $(u,v) \in \mathbb{C}^* \times \mathbb{C}^*$,
we have
\begin{align*}
\Phi \circ \eta_M(u,v) = f_M \circ \Phi(u,v).
\end{align*}
\end{proposition}
We refer the reader to \cite[Section 1.5]{cantat-2} for more details.

Using this semi-conjugacy, it is rather straightforward to show all of the following assertions:
\begin{itemize}
\item[(i)] $\mathcal{J}_4 = S_4$ and $\mathcal{F}_4 = \emptyset$,
\item[(ii)] $\nondiscrete_{4} = \emptyset$ and $\discrete_{4} = S_4$, and
\item[(iii)]  Using Theorem A and Item (i), there exists $p \in S_4$ whose orbit under $\AUTOGROUP_4$ is dense
in $S_4$.
\end{itemize}

\begin{remark}\label{RMK_SADDLE_PTS}

Using the semi-conjugacy $\Phi$ one can show that
the closure of the set of points $q$ for which there exists $\gamma \in \AUTOGROUP \setminus \{\rm id\}$
with $\gamma(q) = q$ and $D\gamma(q)$ hyperbolic coincides with the topological sphere
$\Phi(\mathbb{T}^2) \subseteq S_4 \cap [-2,2]^3$. In particular such
points are not dense in $\mathcal{J}_4 = S_4$.

\end{remark}

\noindent
We refer the reader to \cite[Theorem D]{RR} for more details about Claims (i-iii) above and about the assertion in the remark.

\section{Markoff Parameter $D=0$}
\label{SEC:MARKOFF}
In contrast to the rather well-understood dynamics on the Cayley Cubic $S_4$, the dynamics on the Markoff surface $S_0$
is more mysterious.  One reason for this is that the Fatou/Julia and Locally Discrete/Locally Non-Discrete
dichotomies both become non-trivial.   

\subsection{Fatou/Julia decomposition of $S_0$}
\label{SUBSEC:FJ_MARKOFF}
Before asking questions about the dynamics of $\AUTOGROUP_0$ acting on $S_0$, we
explain why the Fatou set $\mathcal{F}_0$ is non-empty.  The reason goes back
to the works of Bers, Bowditch, and others and it naturally arises through the
interpretation of $S_0$ as a character variety, as explained in Section \ref{SEC:CHARACTERDYNAMICS}.  

\begin{remark} We include in Appendix \ref{SEC:FATOUCOMPS} a brief and lightweight proof that $\mathcal{F}_0 \neq \emptyset$.   It can be
read without learning the material on character varieties.   However, the reader should be advised that
the ideas of that proof all came from the works cited in Sections \ref{SEC:CHARACTERDYNAMICS} and \ref{SEC:MARKOFF} and that a better understanding comes from the
knowledge of its origins within the dynamics of character varieties.
\end{remark}

A Fuchsian representation $\rho: \F_2 \rightarrow \SL(2,\mathbb{R})$ satisfying
$\tr(\rho(aba^{-1}b^{-1})) = -2$ corresponds to a hyperbolic structure on
$\mathbb{T}_0$ with a cusp at the puncture.  By applying quasi-conformal deformations of
such a representation, Bers \cite{BERS} found an embedding of ${\rm Teich}(\mathbb{T}_0) \times {\rm Teich}(\mathbb{T}_0)$
into ${\rm Rep}(\mathbb{T}_0)$, where ${\rm Teich}(\mathbb{T}_0)$ denotes the Teichm\"uller space
of all marked hyperbolic structures on $\mathbb{T}_0$ having a cusp.   This leads to an
open set of Quasi-Fuchsian Representations
\begin{align*}
V_{\rm QF} \subset \chi_{-2} \equiv S_0.
\end{align*}
Recall that a group is Quasi-Fuchsian iff it is a discrete subgroup of ${\rm PSL(2,\C)}$ whose
limit set is a Jordan curve.   Since the action of $\Out(\F_2)$ corresponds to
changing the choice of generators for the group, it does not change either the group itself or its limit set. It follows
that $V_{\rm QF}$ is invariant under the action of $\Out(\F_2)$.

Consider the following two conditions on a point $(x,y,z) \in S_{0}$:
\begin{itemize}
\item[(BQ1)]  None of the coordinates of $\gamma(x,y,z)$ is in $[-2,2]$  for any $\gamma \in \AUTOGROUP$, and
\item[(BQ2)]  $\gamma(x,y,z) \in \Big(\mathbb{C} \setminus \overline{\mathbb{D}_{2}}\Big)^3$ for all but finitely many $\gamma \in \AUTOGROUP$.
\end{itemize}
Here, $\overline{\mathbb{D}_{2}}$ denotes the closed disc of radius $2$ centered at the origin.

The {\em Bowditch BQ Set} is defined as
\begin{align*}
V_{\rm BQ} = \{(x,y,z) \in S_0 \, : \, \mbox{Conditions BQ1 and BQ2  hold}\}.
\end{align*}
Note that $V_{\rm BQ}$ is open, as proved in \cite[Theorem 3.16]{bowditch}.

\begin{proposition}\label{PROP_CONTAINMENTS} We have
$V_{\rm QF} \subseteq V_{\rm BQ} \subseteq \mathcal{F}_0$.
In particular, $\mathcal{F}_0 \neq \emptyset$.
\end{proposition}

\begin{proof}[Sketch of the proof]
The first containment is shown in Section 4 of \cite{bowditch}.
To see the second containment, 
suppose that $p_0 \in V_{\rm BQ}$.  Using that $V_{\rm BQ}$ is open there exists $\epsilon > 0$
such that $\mathbb{B}_\epsilon(p_0) \cap S_0 \subset V_{\rm BQ}$.   
Suppose $p \in \mathbb{B}_\epsilon(p_0) \cap S_0$.
By Condition BQ1 we have
that the $x$-coordinate of $\gamma(p)$ lies in $\C \setminus [-2,2]$ for any $\gamma \in \AUTOGROUP$.  Let $\pi_x(x,y,z) = x$
be the projection onto the first coordinate.  Montel's Theorem
then implies that $\pi_x \circ \AUTOGROUP$ forms a normal family on $\mathbb{B}_\epsilon(p_0) \cap S_0$. 
The same holds for $y$ and $z$ and therefore 
the action of $\AUTOGROUP$ on $\mathbb{B}_\epsilon(p_0) \cap S_0$ forms a normal family.   Since $p_0 \in V_{\rm BQ}$
was arbitrary we conclude that $V_{\rm BQ} \subseteq \mathcal{F}_0$.
\end{proof}

\noindent
{\bf Bowditch Conjecture:} $V_{\rm QF} = V_{\rm BQ}$.

\vspace{0.1in}
\noindent
As explained by Bowditch \cite[p. 702-703]{bowditch} $V_{\rm QF}$ corresponds to four connected components of
$V_{\rm BQ}$.   Therefore, a positive solution to the Bowditch conjecture would consist of proving
that there are no other connected components of $V_{\rm BQ}$.   Note also that part of the boundary of $V_{\rm BQ}$ is the boundary of $V_{\rm QF}$,
which is very irregular; see for example \cite[Appendix A]{mcmullen}.
See \cite[Conjecture A]{bowditch} and also \cite{STY,LEE_XU,series2} for recent related works and more details about the Bowditch Conjecture.

\vspace{0.1in}

\begin{question}\label{Q1}
Does $V_{\rm BQ} = \mathcal{F}_0$?
\end{question}

\noindent
This question serves as a kind of ``complement'' to the Bowditch conjecture.
Note that a positive answer to Question \ref{Q1} would, in part, require showing that there is no 
bounded Fatou component $U$ satisfying
\begin{itemize}
\item[(i)] $U \subset \mathbb{D}_{2}^2$ and
\item[(ii)] $U$ is stabilized by an infinite subgroup of $\AUTOGROUP$.
\end{itemize}

In view of Theorem A, an affirmative answer to Question \ref{Q1} would
immediately imply the existence of a point $p \in S_0 \setminus V_{\rm BQ}$ whose
orbit under $\AUTOGROUP$ is dense in $S_0 \setminus V_{\rm BQ}$.   Moreover, in
\cite[Corollary 5.6]{bowditch}, Bowditch proved that $(0,0,0)$ is an interior
point of $S_0 \setminus V_{\rm BQ}$.  Therefore the orbit closure of such a
point $p$ would contain an open neighborhood of $(0,0,0)$ in $S_0$.

A conditional step in this direction is the following result.

\begin{theorem}\label{THM:CONDITIONAL} Suppose that if $\gamma(q) = q$ for some $\gamma \in \AUTOGROUP \setminus \{\rm id\}$ then
$q \in \mathcal{J}_0$.   Then an open neighborhood of $(0,0,0)$ in $S_0$ is contained in $\mathcal{J}_0$.
In particular, there is a point $p \in \mathcal{J}_0$ whose orbit closure contains an open neighborhood of $(0,0,0)$ in $S_0$.
\end{theorem}

\noindent
This statement is an application of Theorem G from \cite{RR}, after noticing that the
hypothesis in that theorem that ``$(A,B,C,D) \not \in \mathcal{H}$'' is only used to rule out existence
of a point $q \in \mathcal{F}_0$ that is fixed by some $\gamma \in \AUTOGROUP \setminus \{\rm id\}$.

Independently of Theorem \ref{THM:CONDITIONAL}, the following question seems pertinent:

\begin{question}\label{Q2}
Does there exist $p \in S_0$ whose orbit closure contains an open neighborhood of $(0,0,0)$?
\end{question}

\noindent
Meanwhile, the hypothesis of Theorem \ref{THM:CONDITIONAL} leads to another question:

\begin{question}\label{Q3}
Is any point $q \in S_0$ that is fixed by some $\gamma \in \AUTOGROUP \setminus \{\rm id\}$ necessarily
in the Julia set?
\end{question}

\noindent
Let $q$ be a smooth point of $S_{0}$.  An immediate consequence of the
existence of the invariant volume form $\Omega$ (defined in Section
\ref{SEC:INV_FORM}) is that if $q$ is fixed by $\gamma \in \AUTOGROUP \setminus \{\rm id\}$
then ${\rm det}(D\gamma(q)) = 1$.  In particular, if $q$ is a hyperbolic fixed
point for $\gamma$, then it must be of saddle type.  
If $q$ is not hyperbolic then the two eigenvalues of $q$ are
reciprocal, thus leading to resonances for the problem of linearizing $\gamma$
at $q$.   This makes it rather unlikely that such a $q$ could be in
$\mathcal{F}$.   This is a well-know tricky issue that is explained at length
in the paper \cite{mcmullen2} by McMullen, see especially the remark on p. 220
of that paper.

\begin{question}\label{Q4}
Is the set of points $q$ for which there exists $\gamma \in \AUTOGROUP \setminus \{\rm id\}$
with $\gamma(q) = q$ and $D\gamma(q)$ a saddle dense in $\mathcal{J}_0$?
\end{question}

\noindent
Compare with Remark \ref{RMK_SADDLE_PTS} about the Picard Parameter.

\subsection{Locally Non-discrete/Locally Discrete decomposition of $S_0$}
The Locally Non-discrete /Locally Discrete decomposition for $S_0$ is simpler to
describe.  It was proved by Bowditch \cite{bowditch} and later by
Tan-Wong-Zhang \cite{TWZ} that the action of $\AUTOGROUP$ on $V_{\rm BQ}$ is
properly discontinuous.   In particular, $V_{\rm BQ}$ is contained in the
Locally Discrete Locus $\discrete_0$.   

Using a non-linear version of the Zassenhaus Lemma
\cite[Proposition 7.1]{RR}, whose idea dates back to Ghys \cite{GHYS}, we were
able to show that there is an open neighborhood $U$ of $(0,0,0)$ in $S_0$
contained in $\nondiscrete_0$.  (See \cite[Lemma 7.5]{RR}).   We summarize these two paragraphs
by:

\begin{proposition}
The subsets $\nondiscrete_0$ and $V_{\rm BQ}$ of $S_0$ are open, non-empty, disjoint, and $\AUTOGROUP$-invariant.
\end{proposition}

\section{Arbitrary Parameters $(A,B,C,D)$}
\label{SEC:ARBITRARYPARAMS}
For arbitrary parameters $(A,B,C,D) \in \C^4$, the setup explained in Sections \ref{SUBSEC:BACKGROUND_AND_DEFS} and \ref{SEC:INV_FORM} 
remains almost the same, with the mappings $s_x, s_y$ and $s_z: S_{A,B,C,D} \rightarrow S_{A,B,C,D}$ being defined geometrically using lines parallel to the coordinate axes, as before.   However the equations for the mappings now depend on $A,B,$ and $C$:
\begin{align*}
s_x\left(x, y, z\right) &= \left(-x-yz+A,  y, z \right), \\ 
s_x\left(x, y, z\right) &= \left(x, -y-xz + B, z \right), \quad \mbox{and} \\
s_z\left(x, y, z\right) &= \left(x, y, -z-xy + C \right).
\end{align*}
One then defines 
$\AUTOGROUP_{A,B,C,D}^\pm = \langle s_x, s_y, s_z\rangle \leq {\rm Aut}(S_{A,B,C,D})$ and $\AUTOGROUP_{A,B,C,D} = \langle g_x, g_y, g_z\rangle \leq \AUTOGROUP_{A,B,C,D}^\pm$ with $g_x = s_z \circ s_y$, $g_y = s_x \circ s_z$, and $g_z = s_y \circ s_x$, as in Section \ref{SUBSEC:BACKGROUND_AND_DEFS}.

The invariant holomorphic two-form also now depends on $A,B,$ and $C$:
\begin{align*}
\Omega = \frac{dx \wedge dy}{2z+xy-C} = \frac{dy \wedge dz}{2x+yz-A} = \frac{dz \wedge dx}{2y + zx-B}.
\end{align*}
The definitions of the Fatou set $\mathcal{F}_{A,B,C,D}$, the Julia set $\mathcal{J}_{A,B,C,D}$, the Locally Non-discrete set $\nondiscrete_{A,B,C,D}$, and the Locally Discrete set $\discrete_{A,B,C,D}$ are exactly the same as in Sections \ref{SEC:FATOU_JULIA} and \ref{SEC:DISC}.

The case of arbitrary parameters also arises naturally from the dynamics on
character varieties except that one now considers the ${\rm SL}(2,\mathbb{C})$
character variety of the sphere $\mathbb{S}^2_4$ with four punctures (holes).
For a fixed choice of parameters $(A,B,C,D)$ the surface $S_{A,B,C,D}$
corresponds to the relative character varieties with prescribed traces
``assigned to" the small loops around each of the four punctures, analogous to
how  in Section \ref{SEC:CHARACTERDYNAMICS} we ``assigned trace'' $\tau$ to the
unique puncture in~$\mathbb{T}_0$.  (Note that there is a small ambiguity here
since the the parameters $(A,B,C,D)$ do not uniquely determine these four traces; see Equations (\ref{For-A_B_C}) and (\ref{For-D_finally})
below.
Therefore, for a given choice of $(A,B,C,D)$ the surface $S_{(A,B,C,D)}$ corresponds to the character
varieties for up to $24$ different choices of the traces; see \cite[Section 1.4]{cantat-2}.)

The dynamics is induced by the subgroup
${\rm Mod}^\pm_0(\mathbb{S}^2_4)$ of the extended mapping class group ${\rm
Mod}^\pm(\mathbb{S}^2_4)$ consisting of homeomorphisms that preserve each of
the punctures.  There is also an isomorphism from the matrix group
$\MATRIXGP^\pm_2$ to the automorphism group $\AUTOGROUP_{A,B,C,D}^\pm$, which
we will denote by $f \mapsto f_M$, analogous to
Proposition~\ref{PROP:DERIVINGSXSYSZ}.   It is obtained by observing that ${\rm
PGL}(2,\mathbb{Z})$ (and hence $\MATRIXGP^\pm_2$) naturally embeds in ${\rm
Mod}^\pm(\mathbb{S}^2_4)$ and doing similar calculations to those done in the
proof of Proposition~\ref{PROP:DERIVINGSXSYSZ}.  Rather than elaborating
further, we refer the reader to \cite{cantat-2} or \cite{MPT} for more details.

Below we ask several questions about the dynamics of $\AUTOGROUP_{A,B,C,D}$ acting on $S_{A,B,C,D}$ at arbitrary parameters.  They are all organized around the informal question:

\vspace{0.1in}
\noindent
{\em If $(A,B,C,D)$ is not equal to the Picard Parameter $(0,0,0,4)$, to what extent is the dynamics of $\AUTOGROUP_{A,B,C,D}$ acting on $S_{A,B,C,D}$
similar to that for the Markoff Parameter $(0,0,0,0)$?}

\subsection{Bowditch Set:}
Maloni, Palesi, and Tan \cite{MPT} extend the definition of the Bowditch Set to the case
of arbitrary parameters $(A,B,C,D)$ (i.e.\ to the case of the character variety of the four holed sphere)  by adapting the BQ conditions
as follows:
\begin{itemize}
\item[(BQ1)]  None of the coordinates of $\gamma(x,y,z)$ is in $[-2,2]$  for any $\gamma \in \AUTOGROUP$, and
\item[(BQ2)]  $\gamma(x,y,z) \in \Big(\mathbb{C} \setminus \overline{\mathbb{D}_{K}}\Big)^3$ for all but finitely many $\gamma \in \AUTOGROUP$.   
\end{itemize}
Here, $K \equiv K(A,B,C,D)$ is a constant that is made precise in \cite[Definition 3.9]{MPT}.  
Note that the exact value of $K(A,B,C,D)$ does not really matter because replacing it with a larger value will result in the same set $V_{\rm BQ}(A,B,C,D)$.  Note also that in the case of the punctured torus parameters one can use $K(0,0,0,D) = 2$ for 
any $D \in \C$. 

It is proved in \cite{TWZ} and  \cite{MPT} that for any $(A,B,C,D)$ the set
\begin{align*}
V_{\rm BQ} = \{(x,y,z) \in S_{A,B,C,D} \, : \, \mbox{Conditions BQ1 and BQ2  hold}\}
\end{align*}
is open and that $\AUTOGROUP_{A,B,C,D}$ acts properly discontinuously on it.  Hence $V_{\rm BQ}(A,B,C,D)  \subseteq \discrete_{A,B,C,D}$.
The same proof as in Proposition \ref{PROP_CONTAINMENTS} also gives that $V_{\rm BQ}(A,B,C,D) \subseteq
\mathcal{F}_{A,B,C,D}$ for any $(A,B,C,D)$. 

It was proved in \cite{TWZ},  \cite{MPT}, and \cite{GMST} that $V_{\rm BQ}(A,B,C,D) \neq
\emptyset$ for certain families of real (and also purely imaginary) parameters.
It was proved in \cite[Theorem E]{RR} that $V_{\rm BQ}(A,B,C,D)~\neq~\emptyset$ for an open neighborhood of $\{(0,0,0,D) \in \mathbb{C}^4 \, : \, D
\neq 4\}$ in $\mathbb{C}^4$.  In particular, $V_{\rm BQ} \neq \emptyset$ for all punctured torus parameters other than the
Picard Parameter \cite[Corollary to Theorem E]{RR}.
We therefore ask the following question:

\begin{question}\label{Q5}
Is $V_{\rm BQ}(A,B,C,D) \neq \emptyset$ for all $(A,B,C,D) \in \C^4 \setminus \{(0,0,0,4)\}$?
\end{question}

\subsection{Locally Non-Discrete Set}
\begin{question}\label{Q6}
Is the Locally Non-Discrete set $\nondiscrete_{A,B,C,D} \neq \emptyset$ for all $(A,B,C,D) \in \C^4 \setminus \{(0,0,0,4)\}$?
\end{question}

\noindent
In \cite[Theorem F]{RR} we prove that there is an open neighborhood
$\mathcal{P} \subset \C^4$ of the Markoff Parameter $(0,0,0,0)$ and also of
the ``Dubrovin-Mazocco Parameters'' (see \cite[Section 1.2]{RR} for the
definition) such that $\nondiscrete_{A,B,C,D} \neq \emptyset$ for all
$(A,B,C,D) \in \mathcal{P}$.  However, it seems likely that $\nondiscrete_{A,B,C,D} \neq \emptyset$ for a much wider range of parameters.

\subsection{Restating Questions 1-4}
These questions were phrased for the Markoff Parameter $(0,0,0,0)$ in order to highlight them in a case of particular interest.  However, they are all relevant to an arbitrary choice of parameters $(A,B,C,D) \neq (0,0,0,4)$.   In particular, let us ask

\begin{question}\label{Q7}
For any choice of parameter $(A,B,C,D) \in \C^4$,
\begin{itemize}
\item[(a)]  does there exist a point $p$ whose orbit closure in $S_{A,B,C,D}$ has
non-empty interior?
\item[(b)] does the Julia set $\mathcal{J}_{A,B,C,D}$ have non-empty interior?
\end{itemize}
Note that an affirmative answer to {\rm (b)} implies an affirmative answer to {\rm (a)} by Theorem A.
\end{question}

\noindent
We also ask:

\begin{question}\label{Q8}
Can there be a bounded Fatou component for the action of $\AUTOGROUP_{A,B,C,D}$ on $S_{A,B,C,D}$?
\end{question}

\begin{remark}
A bounded Fatou component $V$ for the action of $\AUTOGROUP_{A,B,C,D}$ on
$S_{A,B,C,D}$ cannot be invariant under the whole group $\AUTOGROUP_{A,B,C,D}$
because of the classification of points with bounded orbits given in
\cite[Theorem C]{cantat-2}.  In that theorem Cantat and Loray prove that if $p \in
S_{A,B,C,D}$ has infinite but bounded orbit then $A,B,C,D \in \mathbb{R}$ and
the orbit lies in the real part of the surface $S_{A,B,C,D}(\mathbb{R})$.

Said differently, if $V$ is a bounded Fatou component then its stabilizer
$\AUTOGROUP_V$ must be a proper subgroup of $\AUTOGROUP_{A,B,C,D}$.  More is known: so
long as $V$ does not contain a fixed point of any element of $\AUTOGROUP_{A,B,C,D}
\setminus \{\rm id\}$ then $\AUTOGROUP_V$ must be Abelian (see Theorem K from
\cite{RR} and its proof).   The domain $V$ then wanders under the action of the quotient
$\AUTOGROUP / \AUTOGROUP_V$, with each of the distinct images relatively compact in $S_{A,B,C,D}$.
\end{remark}

The inability to rule out the existence of bounded Fatou components at the
level of generality posed in Question \ref{Q8} is the reason why we did a
``parameter selection" to eliminate countably many hypersurfaces $\mathcal{H}
\subset \C^4$ in the statements of \cite[Theorems G and K]{RR}.   It was done
precisely so that if $(A,B,C,D) \in  \C^4 \setminus \mathcal{H}$ then any fixed
point of any $\gamma \in \AUTOGROUP_{A,B,C,D} \setminus \{\rm id\}$ is in the
Julia set.   Under this hypothesis,  we could then use Theorem K to rule out
bounded Fatou components on a specific open subset $U \subset S_{A,B,C,D}$ on
which we constructed sequences of non-commuting elements of
$\AUTOGROUP_{A,B,C,D}$ converging locally uniformly to the identity.
Existences of such elements would prove that the stabilizer $\AUTOGROUP_V$ of a
hypothetical Fatou component $V$ is non-Abelian, thus contradicting the conclusion
of Theorem K.

\subsection{Invariant Boundary}
There is an open neighborhood $\parameterneighborhood \subset \mathbb{C}^4$ of $(0,0,0,0)$ such that
for every choice of parameter $(A,B,C,D) \in \parameterneighborhood$ both 
$\nondiscrete_{A,B,C,D}$ and $V_{\rm BQ}(A,B,C,D)$ are non-empty.
In particular, there is a set
\begin{align*}
\ourboundary_{A,B,C,D} \subset \partial \nondiscrete_{A,B,C,D} = \partial \discrete_{A,B,C,D}
\end{align*}
that has topological dimension equal to three and is invariant under $\AUTOGROUP_{A,B,C,D}$.
We refer the reader to \cite{FRACTALS} for the definition of topological dimension (used above) and Hausdorff dimension
(used in Question \ref{Q9}, below).

The invariant set $\ourboundary_{A,B,C,D} \subset S_{A,B,C,D}$ of topological
dimension~$3$ ``persists'' over the open subset of parameters
$\parameterneighborhood \subset \mathbb{C}^4$.  The existence of persistent
invariant sets of topological dimension~$3$ for the action of a large group
($\AUTOGROUP$ is free on two generators) strongly hints at a fractal nature for
$\ourboundary_{A,B,C,D}$.

\begin{proposition}\label{PROP:NONDISCRETEREPS}
In any neighborhood 
$W \subset \mathbb{C}^4$ of $(0,0,0,0)$ there are parameters $(A,B,C,D)$
such that the corresponding surface $S_{A,B,C,D}$ contains no (equivalence class of a) discrete representation.
\end{proposition}

\begin{proof}	
A simple way to achieve this is to consider the restricted family $S_{0,0,0,D} \cong \chi_{-2+D}$
and to note that the $D = \tr(\rho(aba^{-1}b^{-1})) + 2$ where the commutator
$aba^{-1}b^{-1}$ corresponds to the class in $\pi_1(\mathbb{T}_0)$ of a small
simple closed loop around the unique puncture (hole) in $\mathbb{T}_0$.  As we increase $D$ 
from $0$ to $\epsilon > 0$ we have that $\tr(\rho(aba^{-1}b^{-1}))$ passes
through values from $-2$ to $-2 + \epsilon$.  For such values of $D$ the
corresponding element $\rho(aba^{-1}b^{-1}) \in \SL(2,\mathbb{R})$ will be
elliptic and it will have an irrational rotation number for almost every such
$D \in [0,\epsilon]$.   For these values of $D$ the representation $\rho$ will be non-discrete for for any choice
of $[\rho] \in \chi_{-2+D} \cong S_{0,0,0,D}$, i.e. for any $(x,y,z) \in  S_{0,0,0,D}$.

If one desires parameters $(A,B,C,D)$ with $A,B,C$ not all equal to $0$ one can do a similar procedure using the interpretation of $S_{A,B,C,D}$
as the character variety of the quadruply punctured sphere.
Given a representation from the fundamental group of the quadruply punctured sphere in $\SL(2,\C)$, let $a_1, \ldots, a_4$ denote the traces
of the matrices associated with small loops around each of the punctures. These traces determine the parameters $A,B,C,D$ by means
of the formulae
\begin{equation}
A = a_1 a_4 + a_2a_3 \, , \; \; B = a_2a_4 + a_1 a_3 \, , \; \; C = a_3a_4 + a_1 a_2 \,  , \label{For-A_B_C}
\end{equation}
and
\begin{equation}
D = 4 - [a_1a_2a_3a_4 + a_1^2 + a_2^2 + a_3^2 + a_4^2 ] \, . \label{For-D_finally}
\end{equation}
It can also be shown that arbitrarily close to the Markoff Parameter $(A,B,C,D) = (0,0,0,0)$ there are
surfaces $S_{A,B,C,D}$ containing no discrete representations. In fact, representations lying in $S_{0,0,0,0}$ are of two
possible natures: either $3$ of the parameters $a_1, \ldots ,a_4$ are zero and the fourth one takes on the values $\pm 2$ or they
are all equal to $\pm \sqrt{2}$ (three of them being equal and the fourth one taking on the value with opposite sign). In particular
they are again generated by elliptic and parabolic elements. Thus by suitably choosing $A,B,C,D$ arbitrarily close to~$0$, we can
ensure that the corresponding representations contain an elliptical element with irrational rotation number corresponding
to the homotopy class of a small loop around at least one of the punctures. These representations will be non
discrete, regardless of the values of $(x,y,z) \in S_{A,B,C,D}$.
\end{proof}

In the case of the Markoff Parameter $(0,0,0,0)$ we have that $V_{\rm QF}
\subseteq \discrete_{(0,0,0,0)}$, with both of them invariant under $\AUTOGROUP_{(0,0,0,0)}$, but we do not know if their boundaries
coincide.  The boundary of $V_{\rm QF}$ is invariant under $\AUTOGROUP_{(0,0,0,0)}$ and is known to be very irregular; see for
example \cite[Appendix A]{mcmullen}.  The point of Proposition
\ref{PROP:NONDISCRETEREPS}, above, is to show that the existence of the
``separating sets'' $\ourboundary_{A,B,C,D}$ and their study lies beyond the
reach of Teichm\"uller theory.   We therefore raise the following question:

\begin{question}\label{Q9}
Does there exist a parameter 
$(A,B,C,D) \in \parameterneighborhood$ such that $S_{A,B,C,D}$ contains no Quasi-Fuchsian representations
and for which the Hausdorff Dimension of $\ourboundary_{A,B,C,D}$ is strictly greater than three?
\end{question}

\noindent
Remark that we do not know of any choice of parameter $(A,B,C,D)$ for which $\ourboundary_{A,B,C,D}$ is smooth.

\subsection{Infinite Index Subgroups}
Recall the isomorphism from $\MATRIXGP^\pm$ to $\AUTOGROUP^\pm$ given by $M \mapsto f_M$ that was described in
Section \ref{SEC:PICARD}.   An element $f_M \in \AUTOGROUP$ is called {\em elliptic, parabolic,} or {\em hyperbolic} (respectively)
if the two-by-two matrix $M$ generates an {\em elliptic, parabolic,} or {\em hyperbolic} (respectively) isometry of the hyperbolic plane $\mathbb{H}$.
For any $(A,B,C,D)$ the group $\AUTOGROUP_{A,B,C,D}$ has the parabolic elements
$g_x, g_y$, and $g_z$ that preserve the fibrations $x={\rm const}$, $y={\rm
const}$, and $z={\rm const}$ respectively.   Moreover, the dynamics of each of
these individual mappings is rather easy to understand and it plays a crucial
role in several of the proofs from \cite{RR}.  (This is especially the case for
Theorems A, B, and C from that paper.)

\begin{question}\label{Q10}
Suppose that $\Upsilon_{A,B,C,D} < \AUTOGROUP_{A,B,C,D}$ is a subgroup that is not
solvable and such that every $\gamma \in \Upsilon_{A,B,C,D} \setminus \{\rm id\}$ is hyperbolic.
\begin{itemize}
\item[(i)] What can be said about the Fatou/Julia and also the Locally Non-Discrete/Locally Discrete dichotomies for
$\Upsilon_{A,B,C,D}$?   
\item[(ii)] What can be said about the closures of $\Upsilon_{A,B,C,D}$ orbits of points $p \in S_{A,B,C,D}$?   Can one
obtain an orbit closure having non-empty interior?
\end{itemize}
\end{question}

\appendix

\section{Proof of Proposition \ref{PROP:DERIVINGSXSYSZ}}
\label{SEC:DERIVINGSXSYSZ}

We have learned the material in this section directly from 
the papers \cite{GOLDMAN,goldman-1,cantat-2,cantat-1}.    However, we hope that
a brief recollection of it here can be useful for the reader.

\begin{proposition}\label{PROP:TRACE_RELATIONS}
For any matrices $M,N \in \SL(2,\mathbb{C})$ we have:
\begin{itemize}\label{PROP:TRACES}
\item[(1)] $\tr M = \left(\tr M^{-1}\right)$
\item[(2)] $\tr \left(MN \right) = \tr\left(NM \right)$, and
\item[(3)] $\tr \left(MN\right) = \left(\tr M\right) \left(\tr N\right) - \tr \left(M N^{-1}\right)$.
\end{itemize}
\end{proposition}
\noindent
We refer the reader to \cite[Section 2.2]{GOLDMAN} for a proof.

\begin{proof}[Proof of Proposition \ref{PROP:DERIVINGSXSYSZ}]
Let us begin with $M_z$ for which the computation is slightly easier.  It corresponds to $\phi \in \Aut(\F_2)$ given by
$\phi(a,b) = (a,b^{-1}) = \phi^{-1}(a,b)$.  We have
\begin{align*}
f_\phi(x,y,z) &= f_\phi \big(\tr(\rho(a)), \tr(\rho(b)), -\tr(\rho(ab)\big)  \\
 &= \big(\tr(\rho(\phi^{-1}(a))), \tr(\rho(\phi^{-1}(b))), -\tr(\rho(\phi^{-1}(ab))\big)   \\
 &= \big(\tr(\rho(a)), \tr(\rho(b^{-1})), -\tr(\rho(ab^{-1}))\big)  \\
 &= \big(\tr(\rho(a)), \tr(\rho(b)), -\tr(\rho(a)) \tr(\rho(b)) + \tr(\rho(ab))\big)  \\
& = \big(x,y,-xy-z\big) = s_z(x,y,z).
\end{align*}
Note that we have used relation (3) from Proposition \ref{PROP:TRACE_RELATIONS} to get from line 3 to line 4 of this equation.

The matrix $M_x$ corresponds to $\phi \in \Aut(\F_2)$ given by
$\phi(a,b) = (a^{-1} b^{-2},b) = \phi^{-1}(a,b)$.  
We have:
\begin{align*}
f_\phi(x,y,z) &= f_\phi \big(\tr(\rho(a)), \tr(\rho(b)), -\tr(\rho(ab)\big)  \\
 &= \big(\tr(\rho(\phi^{-1}(a))), \tr(\rho(\phi^{-1}(b))), -\tr(\rho(\phi^{-1}(ab))\big)   \\
 &= \big(\tr(\rho(a^{-1} b^{-2})), \tr(\rho(b)), -\tr(\rho(a^{-1} b^{-2} b))\big). \\
 &= \big(-yz - x,y,z\big) = s_x(x,y,z).
\end{align*}
To explain the last equal sign let $A = \rho(a)$ and $B = \rho(b)$ and note that
\begin{align*}
\tr \left(A^{-1}B^{-2}\right) &= \left(\tr(A^{-1}B^{-1})\right) \left(\tr B^{-1}\right) - \tr \left(A^{-1}B^{-1} B \right) = \left(\tr(B^{-1}A^{-1})\right)\left(\tr B^{-1}\right) - \tr(A^{-1}) \\ &= \tr \left(AB\right)\tr(B) - \tr(A) = -zy - x \qquad \mbox{and} \\
\tr(A^{-1}B^{-1}) &= \tr(B^{-1}A^{-1}) - \tr((AB)^{-1}) = \tr(AB) = z.
\end{align*}
The calculation for $M_y$ is quite similar and left to the reader.
\end{proof}

\section{Fatou set for the Markoff Parameter}
\label{SEC:FATOUCOMPS}

\begin{proposition} The Fatou set ${\mathcal F}_0$ is non-empty for the Markoff Parameter $D=0$.
\end{proposition}

This proof is an adaptation of ideas that we learned when reading
\cite{Hu}:

\begin{proof}
Let $\mathbb{B}_{1/4}(p_0) \subset
\C^3$ denote the Euclidean Ball of radius $1/4$ centered at $p_0 = (-3,-3,-3)$. 
Since $p_0 \in S_0$ we have that
$\mathbb{B}_{1/4}(p_0) \cap S_0$ forms a non-empty open subset of
$S_0$.   We will show that for any
$p \in \mathbb{B}_{1/4}(p_0)$, any integer $k \geq 1$, and any reduced word $w_k w_{k-1} \ldots w_1$ in
the mappings $s_x, s_y$, and $s_z$ that each
coordinate of $w_k w_{k-1} \ldots w_1(p)$ has modulus at least as large as the
corresponding coordinate for $w_{k-1} \ldots w_1(p)$.

In particular, this will 
imply that 
\begin{align*}
w_k w_{k-1} \ldots w_1(p) \in \Big(\mathbb{C} \setminus \overline{\mathbb{D}_{2}}\Big)^3
\end{align*}
for any such word of any length $k \geq 1$.  Applying Montel's Theorem to each
of the three coordinates separately, this will prove that $\AUTOGROUP^\pm$ (and hence
$\AUTOGROUP < \AUTOGROUP^*$) forms a normal family on $\mathbb{B}_{1/4}(p_0)
\cap S_0$ and, in particular, that $\mathbb{B}_{1/4}(p_0) \cap S_0$ lies in $\mathcal{F}_0$.

We first check that our claim holds for $k = 1$ and, without loss
of generality, we can suppose $w_1 = s_x$.  Since $s_x(x,y,z) = (-x-yz,y,z)$
it suffices to check that the $x$-coordinate is not decreased in modulus:
\begin{align*}
|-x-yz| \geq |yz| - |z| = \left(3-\frac{1}{4}\right)^2 - \left(3+\frac{1}{4}\right) = \frac{69}{16} > 3+\frac{1}{4} > |x|.
\end{align*}
Now suppose that there is some integer $j \geq 1$ and some word $w_{j+1} w_j w_{j-1} \ldots w_1$ so that 
for some $p \in \mathbb{B}_{1/4}(p_0)$ one of the coordinates of $w_{j+1} w_j w_{j-1} \ldots w_1(p)$
has modulus less than the corresponding coordinate of $w_j w_{j-1} \ldots w_1(p)$.  Without loss
of generality we can assume that $j$ is the minimal such integer.
This implies that each coordinate of $w_j w_{j-1} \ldots w_1(p)$
has modulus at least as large as the corresponding coordinate of $p$ (hence modulus at least
$11/4 = 3-1/4$).   Let:
\begin{align*}
(x,y,z) = w_{j-1} \ldots w_{1}(p), \qquad (x',y',z') = w_j \ldots w_{1}(p), \quad \mbox{and} \quad (x'',y'',z'') = w_{j+1} \ldots w_{1}(p).
\end{align*}   
As explained above, $|x'|, |y'|, |z'| > 11/4$.

Without loss of generality we can assume $w_j = s_x$ and $w_{j+1} = s_y$ so that
\begin{align*}
(x',y',z') = (-x-yz,y,z) \qquad \mbox{and} \qquad (x'',y'',z'') = (x',-y'-x'z',z').
\end{align*}
Our assumption that $j$ was minimal implies that $|x| \leq |x'|$ 
and that $|y'| > |y''|$.   The first of these inequalities gives us (\ref{FATOU1}), below,
and the second of these inequalities gives us (\ref{FATOU2}), also below.
\begin{align}
2 |x'| &\geq |x'+x| = |-yz| = |y'z'|, \quad \mbox{and}  \label{FATOU1}\\
2 |y'| &> |y'+y''| = |-x'z'| = |x'z'|.   \label{FATOU2}
\end{align}
Suppose first that $|x'| < |y'|$.   In this case (\ref{FATOU1}) implies
\begin{align*}
|y'z'| \leq 2|x'| < 2|y'|.
\end{align*}
Since $|y'| \neq 0$ this gives $|z'| < 2$ which contradicts $|z'| > 11/4$.

Now suppose that $|x'| \geq |y'|$.   In this case (\ref{FATOU2}) implies
\begin{align*}
|x'z'| < 2|y'| \leq 2|x'|.
\end{align*}
Since $|x'| \neq 0$ this again gives $|z'| < 2$ which contradicts $|z'| > 11/4$.
\end{proof}

\vspace{0.1in}
\noindent
{\bf Acknowledgments:}
We thank Eric Bedford, Philip Boalch, Serge Cantat, Jeffrey Diller, Romain
Dujardin, Simion Filip, Yan Mary He, John Hubbard, Bernard Julia, Dan Margalit, Mikhail
Lyubich, Bogdan Nica, Emmanuel Paul, Jean-Pierre Ramis, and Ser Peow Tan for interesting
comments and discussions that have helped us to shape the questions posed in
this work.  We also thank the anonymous referee for his or her detailed comments which
helped us to substantially improve the exposition.  This work was supported by the US National Science Foundation
through grant DMS-2154414 and by the Centre International de Math\'ematiques et
Informatique CIMI through the project ``Complex dynamics of group actions,
Halphen and Painlev\'e systems''.


\vspace{0.2in}

{\footnotesize
\noindent
Julio Rebelo \\
Institut de Math\'ematiques de Toulouse; UMR 5219, \\
Universit\'e de Toulouse, \\
 118 Route de Narbonne,\\
F-31062 Toulouse, France.\\
rebelo@math.univ-toulouse.fr.

\vspace{0.1in}

\noindent
Roland Roeder\\
 Department of Mathematical Sciences\\
Indiana University--Purdue University Indianapolis,\\
Indianapolis, IN, United States.\\
roederr@iupui.edu
}

\end{document}